\documentclass[11pt]{amsart}
\usepackage{amssymb, amsmath, amsthm, amsfonts, amscd}
\usepackage{xcolor}
\usepackage{datetime}
\usepackage{enumerate}
\newtheorem{theorem}{Theorem}[section]
\newtheorem{lemma}[theorem]{Lemma}
\newtheorem{proposition}[theorem]{Proposition}

\theoremstyle{definition}

\theoremstyle{remark}

\numberwithin{equation}{section}

	\title[Approximations of Fixed Points for  Nonexpansive Mappings]{A Note on Approximations of Fixed Points for  Nonexpansive Mappings in Norm-attainable Classes}
	\author{Benard Okelo}
	\address{Institute of Mathematics, University of Muenster, Einstein Street 62, 48149-Muenster, Germany.}
	\email{bnyaare@yahoo.com}

	\date{\currenttime ;  \today}
	
\begin{document}

\begin{abstract}
Let $H$ be an infinite dimensional, reflexive, separable Hilbert space and $NA(H)$ the class of all norm-attainble operators on $H.$ In this note, we study an implicit scheme for  a canonical representation of nonexpansive contractions in norm-attainable classes.
\end{abstract}
\maketitle

\section{Introduction}
\noindent A lot of studies on fixed point theory in convex sets have been carried out with interesting results obtained (see \cite{oke} and \cite{sua2}).  The interesting aspect is that, provided the
existence of a fixed point of a given mapping has been found, what remains is to determine the value
of that fixed point which is not a trivial task \cite{sua2}. This is the reason  why  iterative processes are put into action for
computing them.  The  Banach contraction theorem \cite{Hal} utilized Picard iteration process in approximating a fixed point. In this paper, we give an  implicit scheme for  a canonical representation of nonexpansive contractions in norm-attainable classes.  Here is the main theorem of this work.
\begin{theorem}\label{g1}
 Let $\mathcal{Q}$ be a  two-sided maximal ideal of   a real  separable  Hilbert space $H$ and   $H_{0}$ be a reflexive invariant subspace of $H$.  Suppose that   $ P=\{T_{s}:s\in \mathcal{Q}\}$ is a canonical representation of $\mathcal{Q}$  from $H$ into itself  such that essential closure of  $\lbrace T_{t}x : t \in \mathcal{Q} \rbrace $ is sequentially compact for every $ x \in H_{0} $ and  $  \Xi(\mathcal{Q})\neq \emptyset$.
   Suppose that $X$ is an invariant subspace of $NA(\mathcal{Q})$ such that $1\in X$,  $t\mapsto \langle T_{t}x,x^{*}\rangle$ is an element of $X$ for each $x\in H_{0}$ and $x^{*}\in H^{*}$. Consider $\{\gamma_{n}\}$ as monotone increasing sequence of   $X$.  Suppose that $f$ is a contraction on $ H_{0}$. Let $\epsilon_{n}$ be a sequence in $(0, 1)$ such that $\displaystyle \lim_{n} \epsilon_{n}=0$.  Consider the duality mapping $J$ to be  weakly sequentially continuous.
 Then    we have a unique nonexpansive retraction $ R $  of $ H_{0}$ onto $   \Xi(\mathcal{Q})$ and $ x \in H_{0}$ such that the  sequence  $\{\xi_{n}\}$   generated by
    $\xi_{n}=\epsilon_{n} f(\xi_{n})+(1-\epsilon_{n})T_{\xi_{n}}\xi_{n},
$ is strongly convergent to $Rx$, or all $n\in \mathbb{N}$.
\end{theorem}
\noindent We note that an operator $S\in B(H)$ is said to be norm-attainable if there exists a unit vector $x\in H$ such that $\|Sx\|=\|S\|.$ We denote the class of all norm-attainble operators on $H$ by $NA(H)$. For $S\in NA(H),$ we call $x$ a fixed point of $S$ if $S(x) = x.$ If $H_{0}$ is an invariant subspace of $H$ then $S\in NA(H)$ is said to be nonexpansive if $\|S x - S y\| = \|x - y \|$, for
all $x, y \in H_{0}.$ In the next section we give the main results in this note.
 \section{Main results}
\noindent Now we give the main results of this study.
First, we begin with some auxiliary results and we express our conditions  by removing the compactness condition on an invariant subspace $H_{0}$ of   an infinite dimensional, reflexive, separable Hilbert space $H$ as follows.

\begin{proposition}\label{lmt}
Let $H_{0}$ be an invariant subspace of $H.$ For every   two-sided maximal ideal $\mathcal{Q}$ of  an infinite dimensional, reflexive, separable Hilbert space $H$, there exists a unique nonexpansive retraction $\Xi(\mathcal{Q}) $ of $H_{0}.$
\end{proposition}

\begin{proof}
Since $H$ has the Opial's condition then by Banach contraction mapping theorem, we can fix
a sequence  $\xi_{n}$ in $H_{0}$ i.e.
$
      \xi_{n}=\frac{1}{n}x+(1- \frac{1}{n})T_{\xi}\xi_{n}\quad ( n \in \mathbb{N}),
$
where  $x \in H_{0}$ is fixed and   $\xi$ is an invariant  on $X$.
 From \cite{oke}, we have
 $
  \lim_{n\rightarrow \infty}\|\xi_{n}-T_{\xi}\xi_{n}\|=0.
 $
The boundedness of  $\{\xi_{n}\}$ is trivial so we
that it weakly converges   to an element of $\Xi(\mathcal{Q}) $. That is, we prove that the weak  limit set of $\{\xi_{n}\}$ denoted by $\omega_{\omega}\{\xi_{n}\}$ is contained in $\Xi(\mathcal{Q})$.
  Let $x^{*} \in \omega_{\omega}\{\xi_{n}\} $ and consider $\{\xi_{n_{j}}\}$ be a subsequence of $ \{\xi_{n}\} $ such that $\xi_{n_{j}} \rightharpoonup x^{*}$. Since  $I -T_{t}$ is semiclosed at diminishing point, for each $t \in \mathcal{Q}$, then we conclude that $x^{*} \in \Xi(\mathcal{Q})$. Therefore,  $\omega_{\omega}\{z_{n}\} \subseteq \Xi(\mathcal{Q})$.
Now  $ \{\xi_{n}\} $ is bounded  and    $H$ is separable,   so $ \{\xi_{n}\} $  is a sequentially compact subset of $H$, hence   we have  $ \{\xi_{n_{j}}\}$
 of $ \{\xi_{n}\}$  such that
 $\{\xi_{n_{j}}\}$ sequentially converges  to a point $\xi$.
Invoking  nonexpansivity of retractions of $H_{0}$ onto $\Xi(\mathcal{Q}),$ uniqueness is proved and the proof is complete.
\end{proof}

\noindent The next result is an analogy of  of known assertions whereby we remove the compactness condition on $H_{0}$ for  reflexive real  separable Hilbert spaces  as follows.
\begin{lemma}\label{sunyn}
For every   two-sided maximal ideal $Q$ of    an infinite dimensional, reflexive, separable Hilbert space $H$,  let $X$ be a left invariant subspace of $NA(Q)$ such that $1\in X$, and the function $t\mapsto \langle T_{t}x,x^{*}\rangle$ is an element of $X$ for each $x\in H_{0}$ and $x^{*}\in H^{*}$. Then there exists a unique nonexpansive retraction $\Xi(\mathcal{Q}) $ of $H_{0}.$

\end{lemma}
\begin{proof}
Since $H$ has the Opial's condition,  let   $x \in H_{0}$.
Fix $D=\{y \in H_{0}: \|y-p\|\leq \parallel x-p \parallel\}$. Since $D$ reflexive and  $x \in D$ we have $T_{t}(D)\subset D$, Given
$\epsilon>0,$ the rest of the proof follows immediately from  Proposition \ref{lmt}.
\end{proof}

\noindent Now we dedicate our efforts to proving Theorem \ref{g1}.

\begin{proof}
We proceed with the \textbf{Proof of Theorem} \ref{g1} as follows. We know that  $H$ has the Opial's condition and that $H$ is a reflexive real  separable  Hilbert space.
We give our proof in  five  steps as illustrated below.\\
\noindent  $Step (i).$  Existence of $\{\xi_{n}\}$. This is guaranteed from Proposition \ref{lmt} and also follows  Step $1$ of   \cite{oke}.\\
\noindent  $Step (ii).$  $\{\xi_{n}\}$ is bounded. To see this, let  $p\in  \Xi(\mathcal{Q}) $. Since   $T_{\xi_{n}}p = p$ for each $n \in \mathbb{N}$, by simple calculation we have
$
   \|\xi_{n}-p\|^{2}
                 \leq\epsilon_{n}  \alpha\|\xi_{n}-p\|^{2}+(1-\epsilon_{n})\|\xi_{n}-p\|^{2}+\epsilon_{n}\Big \langle  f(p)-p\:,\:J(\xi_{n}-p)\Big\rangle \leq\frac{1}{1-\alpha}\left< f(p)-p\:,\:J(z_{n}-p)\right>.
$
So,
$
   \|z_{n}-p\|\leq\frac{1}{1-\alpha}\| f(p)-p\|,
$
proving the boundedness of $ \{\xi_{n}\}$.\\
\noindent  $Step (iii).$ We show that  $\lim_{n\rightarrow \infty}\|\xi_{n}-T_{t}z_{n}\|=0 $,  for all $t \in \mathcal{Q}.$
To see this, consider  $t \in \mathcal{Q}$.
Let $p$  be an arbitrary element of $\Xi(\mathcal{Q})$. Set
$D=\{y \in H_{0}: \|y-p\|\leq \frac{1}{1-\alpha}\| f(p)-p\|\}$. From the statement of the theorem, we know that  that $D$ is reflexive and  $\{\xi_{n}\}\subset D$ and $T_{t}(D)\subset D$. The rest of the proof follows from step (vii) in \cite{sua2}.\\
 \noindent  $Step (iv).$   We have a unique retraction $R$  of $ H_{0} $ onto $  \Xi(\mathcal{Q})$ and $ x \in H_{0}$ such that
$
  \Gamma  :=\limsup_{n}\langle x-Rx\,,\,J(\xi_{n}-Rx)\rangle\leq0.
$
The proof follows from the definition of
 $ \Gamma $  and      from Step $(ii)$ that $ \{\xi_{n}\} $ is bounded, and  since $H$ is reflexive and separable, so we have $ \{\xi_{n_{j}}\}$
 of $ \{\xi_{n}\}$ satisfying
 $\displaystyle\lim_{j}\langle x-Rx\,,\,J(z_{n_{j}}-Rx)\rangle=\Gamma$ and
 $\{\xi_{n_{j}}\}$ sequentially converges  to a point $\xi$. By considering Lemma \ref{sunyn}, and the definition of   nonexpansive mapping the rest is clear.\\
\noindent  $Step (v).$ $ \{\xi_{n}\}$  strongly converges to $Rx$. Indeed, this follows from the fact that $H$ is separable and
$
    \limsup_{n}\|\xi_{n}-Rx\|^{2}\leq   \frac{2}{1-   \alpha}\limsup_{n}\langle x-Rx \,,\,J(\xi_{n}-Rx)\rangle\leq0.
$
By simple manipulation it os easy to see that $\xi_{n} \rightarrow Rx$.
\end{proof}
\section*{Acknowledgements}
\noindent The author is grateful to DFG-Germany for the financial support through Research Grant No. DFG-2018-0098

\end{document}